\documentclass[11pt,reqno]{amsart}

\usepackage{amsmath,amssymb,amsthm, amsfonts}   
\usepackage{bbm} 
\usepackage{graphicx}   
\usepackage{pstricks}   
\usepackage[colorlinks=true, linkcolor=blue, citecolor=magenta]{hyperref} 
\usepackage{url}
\usepackage{tasks}
\usepackage{enumerate}
\usepackage{caption,subcaption} 

\usepackage{mathtools}      
\mathtoolsset{showonlyrefs} 

\usepackage{mathrsfs}
\usepackage{multicol}   
\usepackage{fullpage}   

\usepackage[foot]{amsaddr}

\newtheorem{thm}{Theorem}[section]

\newtheorem{lem}[thm]{Lemma}

\newtheorem*{rmk*}{Remark}
\newtheorem*{acknowledgements*}{Acknowledgements}

\newtheorem{thmMain}{Theorem}

\theoremstyle{definition}

\title{The Talbot effect as the fundamental solution to the free Schr\"odinger equation}
\author[D. Eceizabarrena]{Daniel Eceizabarrena}
\address{Department of Mathematics and Statistics, University of Massachusetts Amherst, Amherst, MA 01003-9305, USA.
}
\email{\tt eceizabarrena@math.umass.edu}


\begin{document}

\begin{abstract}
The Talbot effect is usually modeled using the Helmholtz equation, but its main experimental features are captured by the solution to the free Schr\"odinger equation with the Dirac comb as initial datum. 
This simplified description is a consequence of the paraxial approximation in geometric optics. However, it is a heuristic approximation that is not mathematically well justified, so K. I. Oskolkov raised in \cite{Oskolkov2006} the problem of ``mathematizing'' it. We show that it holds exactly in the sense of distributions. 
\end{abstract}

\maketitle


\section{Introduction}

\subsection{The Talbot effect}
In 1836, the English polymath and pioneer in photography Henry Fox Talbot made a ray of light cross \textit{``an equidistant grating with several parallel slits cut''} in a dark chamber, and he observed the result with a magnifying lens \cite{Talbot1836}. Moving the lens back, the illuminated pattern of the grating blurred and recovered neatness alternately. He also remarked the perfect distinction of the  bands, even if the grating was greatly out of the focus of the lens.
This surprising phenomenon, which today we call the Talbot effect, was mathematically studied in 1881 by Lord Rayleigh \cite{Rayleigh1881}, who computed the distance $z_T = d^2/\lambda$ where an image identical to the grating was formed, given the separation of the slits $d$ and the wavelength of monochromatic light $\lambda$. Further experimental findings did not come until the 1960s \cite{HiedemannBreazeale1959,WinthropWorthington1965}, while a more rigorous theoretical analysis was started in the 1990s \cite{BerryGoldberg1988,BerryKlein1996}. We now know that in every rational multiple $z_T p/q$ of Rayleigh's distance, the grating is reproduced with a separation $q$ times smaller. The phenomenon is visually represented by the Talbot carpet in Figure~\ref{Figure_Talbot_Young}. 
A clarifying and interactive simulation of the experiment is available online in \cite{QuantumInteractive}. See also \cite{BerryMarzoliSchleich2001} for further reference.

\begin{figure}[h]
\begin{subfigure}{0.49\linewidth}
\includegraphics[width=\linewidth]{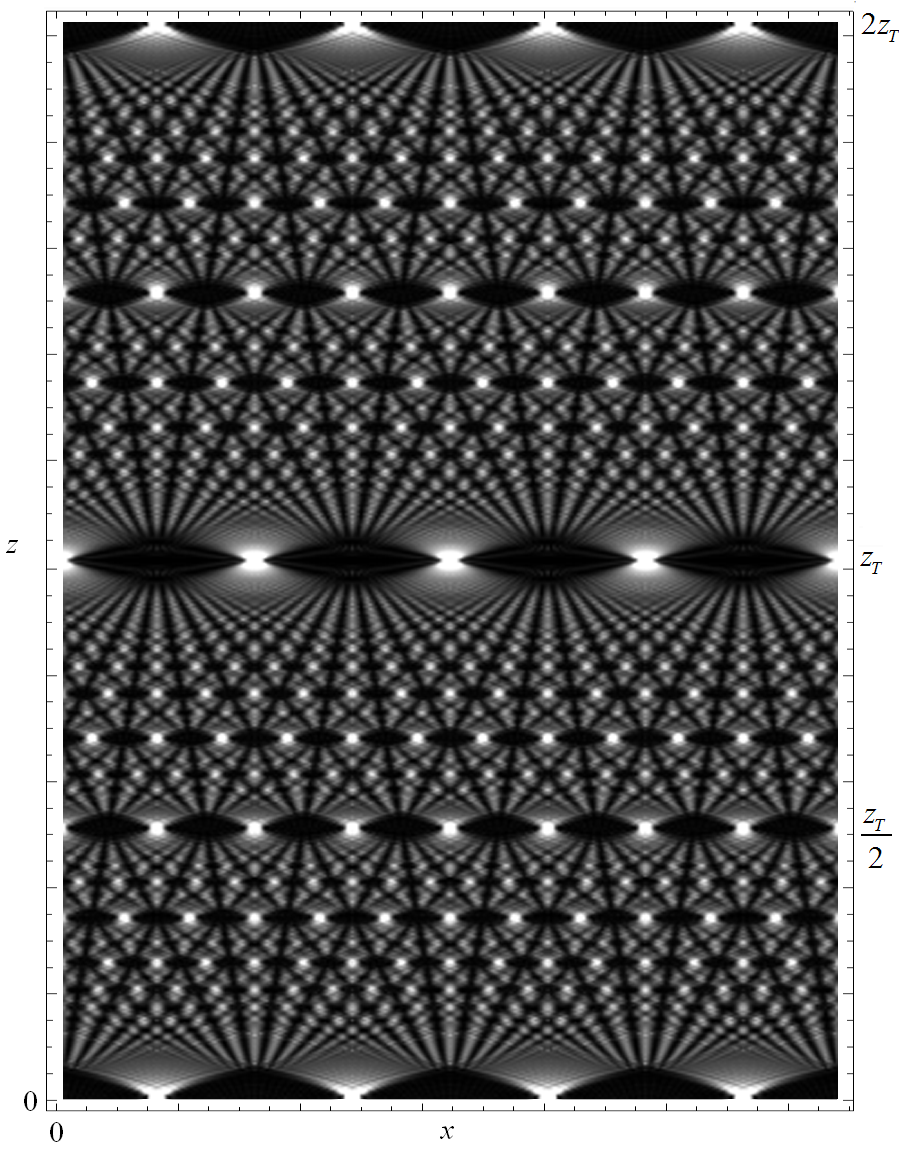}
\caption{The Talbot carpet, a modern visual representation of the Talbot effect. Modified version of original image by Ben Goodman, CC-BY-SA-3.0 \url{https://commons.wikimedia.org/wiki/File:Optical_Talbot_Carpet.png}.}
\label{Figure_TalbotCarpet}
\end{subfigure}
\begin{subfigure}{0.49\linewidth}
\includegraphics[width=0.96\linewidth]{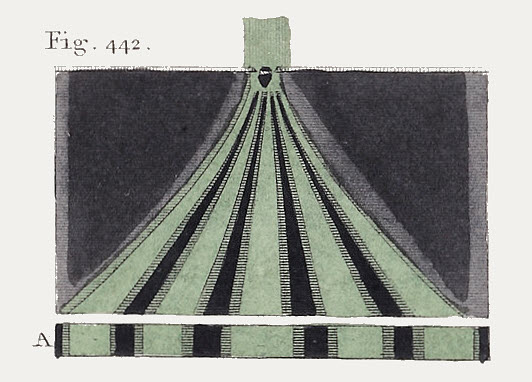}
\caption{The carpet corresponding to Young's celebrated double slit experiment, drawn by Young himself \cite[Plate XXX, Fig. 442]{Young1807}.}
\label{Figure_YoungCarpet}
\end{subfigure}
\caption{Comparison between the Talbot effect and Young's double slit experiment. The experiments are similar, the difference being the number of slits in the grating.}
\label{Figure_Talbot_Young}
\end{figure}

\subsection{The Talbot effect and the Helmholtz equation}
In a two dimensional simplification, the Talbot effect is described by the Helmholtz equation
\begin{equation}\label{Helmholtz_Equation}
\Delta_{x,z} u + k^2 u = 0, \qquad  u = u(x,z),\quad  x \in \mathbb R, \quad z > 0,
\end{equation}
and the grating, which is assumed to have period $d$, is modeled as the boundary datum $u(x,0)$. Since  Rayleigh also showed periodicity of period $z_T$ in $z$, it is convenient to rescale the variables by $\xi = x/d$ and $\zeta = z/z_T$. There are different models for the grating at $\zeta=0$. For instance, in \cite{BerryKlein1996}, Ronchi gratings are proposed, with slits represented by a characteristic function of width $1/2$ centered in the origin and periodically extended. In general, if a periodic grating is represented by its Fourier transform $g(\xi) = \sum_{n \in \mathbb Z} \widehat{g}_n e^{2\pi i n \xi}$, they showed that the mathematical study reduces to the ideal case of the Dirac comb
\begin{equation}\label{Dirac_Comb}
u(\xi,0) = \sum_{n \in \mathbb Z} \delta(\xi - n) = \sum_{n \in  \mathbb Z} e^{2\pi i n \xi}.
\end{equation}
The second identity in \eqref{Dirac_Comb} is a consequence of the Poisson summation formula. Formally, this represents infinitely many infinitesimal slits, which is only an approximation of the physical experiment since such a grating is unreal. However, it can be taken as the fundamental datum for the problem, in parallelism to the fundamental solutions to PDEs. In this article we focus on it, more general data $g$ will be considered in a future work.

The Helmholtz equation \eqref{Helmholtz_Equation} with initial datum \eqref{Dirac_Comb} is solved by
\begin{equation}\label{Helmholtz_Solution}
u(\xi,\zeta) = \sum_{n \in \mathbb{Z}} e^{ 2\pi i (d/\lambda)^2   \sqrt{1-(\lambda n/d)^2} \, \zeta  }\,e^{2\pi i n \xi}, \qquad \xi \in \mathbb R, \quad  \zeta > 0,
\end{equation}
where $k=2\pi/\lambda$. Writing the Laplacian in $(\xi,\zeta)$ as $\Delta_{x,y} = d^{-2}\partial_{\xi\xi} + z_T^{-2} \partial_{\zeta\zeta}$, \eqref{Helmholtz_Solution} comes from testing solutions of the type $u(\xi,\zeta) = \sum_{n\in \mathbb Z} \eta_n(\zeta) \, e^{2\pi i n \xi}$. Also, a choice of the signs in the phase is needed to discard non-physical exponentially growing waves. For the same reason, it can be extended to $\zeta<0$ by setting $u(\xi,-\zeta) = u(\xi,\zeta)$. We refer to \cite{BerryKlein1996,MatsutaniOnishi2003} for details.

\subsection{The Schr\"odinger equation, the paraxial approximation and its mathematization}
The solution \eqref{Helmholtz_Solution} represents the sum of planar waves formed as a result of  diffraction after light crosses the grating. A way to analyze it is to assume that those waves travel paraxially, that is, almost parallel to the $\zeta$ axis. That means assuming that only the waves with $|n| \ll d/\lambda$ in \eqref{Helmholtz_Solution} are relevant, larger $n$ representing waves that either have a direction too bent or that quickly fade away. In that case, one may use the quadratic approximation 
\begin{equation}\label{Paraxial_Approximation_Prelim}
\sqrt{1-x^2} \approx 1-x^2/2
\end{equation}
to get 
\begin{equation}\label{Approximation_By_Schrodinger}
u(\xi,\zeta) \approx e^{2\pi i (d/\lambda)^2 \zeta} \, \sum_{n \in \mathbb \mathbb Z} e^{2\pi i n \xi - i \pi n^2 \zeta}.
\end{equation}
The expression
\begin{equation}\label{Free_Schrodinger_Solution}
v(\xi,\zeta) = \sum_{n \in \mathbb{Z}} e^{2\pi i n \xi - i \pi  n^2\zeta},
\end{equation}
which according to \eqref{Approximation_By_Schrodinger} satisfies $|u(\xi,\zeta)| \approx |v(\xi,\zeta)|$, is the solution to the free Schr\"odinger equation
\begin{equation}\label{Free_Schrodinger_Equation}
v_\zeta = \frac{i}{4\pi} \, v_{\xi\xi}
\end{equation}
with initial/boundary datum as in \eqref{Dirac_Comb}. That this expression captures the main properties of the Talbot effect 
in Figure~\ref{Figure_TalbotCarpet} is a consequence of some basic arithmetic manipulations. Indeed, for any  coprime $p \in \mathbb Z$ and $q \in \mathbb N$,  one can write
\begin{equation}\label{Talbot_By_Schrodinger_Short}
v(\xi,p/q)  =  \frac{1}{q}\,  \sum_{m\in \mathbb Z} \Gamma(p,q;m)  \delta\left(\xi - \frac{p \, (\operatorname{mod} 2)}{2} - \frac{m}{q}\right),
\end{equation}
where 
\begin{equation}
\Gamma(p,q;m) = \sum_{r=0}^{q-1} e^{2\pi i \, \frac{  ( q\, p\text{(mod} 2) / 2 + m )  r  - (p/2) r^2  }{q} }
\end{equation}
is a variation of the classical generalized quadratic Gauss sums. As we briefly explain in Appendix~\ref{Section_GaussSums}, we have $|\Gamma(p,q;m)|= \sqrt{q}$ for all $p,q,m$. It is apparent that the support of \eqref{Talbot_By_Schrodinger_Short} coincides with the pattern in Figure~\ref{Figure_TalbotCarpet}, since at $\zeta=p/q$ there are $q$ times more deltas than in the grating $v(\xi,0)$. For the sake of completeness, the reader can find the details on how to derive \eqref{Talbot_By_Schrodinger_Short} in Appendix~\ref{AppendixTalbot}.

The paraxial approximation \eqref{Approximation_By_Schrodinger} and its implications in the Talbot effect were studied by Berry and Klein in \cite{BerryKlein1996}, also with gratings with finitely many slits and of finite width. Also, to measure its limitations, a \textit{post-paraxial} approximation of the square root in \eqref{Helmholtz_Solution} by $\sqrt{1-x^2} \approx 1-x^2/2 - x^4/8$ was studied, and blurring effects not captured by the paraxial method appeared. The effect of further terms in the Taylor series was claimed to be insignificant when compared to experiments with the gratings that are interesting for the Talbot effect.
Thus, from an experimental perspective, the paraxial approximation yields a faithful representation of the Talbot effect up to some blurring. 
However, this approximation is not properly justified from a mathematical perspective, since clearly the paraxial approximation cannot be applied in \eqref{Helmholtz_Solution} for all $n \in \mathbb Z$. This was noted by  Oskolkov \cite[p.200]{Oskolkov2006}, who proposed to work on a ``mathematization" of these arguments. The objective of this article is to tackle this question.




\subsection{Why to study the Talbot effect mathematically?}\label{SECTION_MathematicalWork}

The description of the Talbot effect by means of the Schr\"odinger equation has physical consequences unrelated to the original optical experiment that make it subject to a mathematical study.  For instance, the phenomenon of quantum revivals is due to the Schr\"odinger description. Like in the optical Talbot effect, a quantum wave-packet is repeatedly reconstructed at fractional multiples of a \textit{quantum period} that depends on the Planck constant and on the energy levels \cite{BerryMarzoliSchleich2001}. One more surprising consequence 
is that the Talbot effect appears in the motion of polygonal vortex filaments, as it was shown in \cite{DelaHozVega2014,JerrardSmets2015} by modeling their evolution using the vortex filament equation and supported by experiments \cite{KlecknerScheelerIrvine2014} and numeric simulations \cite{KumarVideos}.
The mathematical reason behind this is that the vortex filament equation can be transformed to a cubic nonlinear Schr\"odinger equation. The Dirac comb \eqref{Dirac_Comb} as the datum arises interpreting the curvature of the polygon as a periodic sum of Dirac deltas.

In a more abstract setting,
analysts have become interested in the Talbot effect due to its implications in the study of the Schr\"odinger equation. 
The main motivation for this seems to have been Berry and Klein's conjectures on the fractality of the graphs of $\operatorname{Re} v_g(\xi,\zeta)$, $\operatorname{Im} v_g(\xi,\zeta)$ and $|v_g(\xi,\zeta)|^2$, where as above $v_g(\xi,\zeta)=\sum_{n}\widehat g_n e^{2\pi i n \xi - i\pi n^2 \zeta}$ correspond to different gratings \cite{Berry1996,BerryKlein1996}. Fractality was conjectured along lines with almost any fixed $\xi$ or with fixed irrational $\zeta$, or even along oblique lines, with a different dimension in each case. 
Corresponding mathematically rigorous results were proved in \cite{KapitanskiRodnianski1999,Oskolkov1992,Oskolkov2006,Rodnianski2000,Taylor2003}. Moreover,  the Talbot effect and these fractal considerations have also been observed experimentally in a nonlinear setting \cite{ZhangWenZhuXiao2010} and numerically for linear and non-linear dispersive equations \cite{ChenOlver2013,ChenOlver2014,Olver2010},
and some of them have been proved rigorously  
 \cite{ChousionisErdoganTzirakis2015,ErdoganShakan2019,ErdoganTzirakis2013}.
Also in this spirit, variations of the solution \eqref{Free_Schrodinger_Solution} like  Riemann's non-differentiable function \cite{OskolkovChakhkiev2010} and a discrete Hilbert transform \cite{OskolkovChakhkiev2013} have also been studied along lines with fixed space or time.

 \section{Statement of results}\label{Section_Results}

From \eqref{Helmholtz_Solution} and \eqref{Approximation_By_Schrodinger}, let us define the adapted Helmholtz solution
\begin{equation}\label{Candidate_For_Convergence}
w(\xi,\zeta) =  e^{ - 2\pi i (d/\lambda)^2  \zeta } \, u(\xi, \zeta) = e^{ - 2\pi i (d/\lambda)^2  \zeta } \, \sum_{n \in \mathbb{Z}} e^{ 2\pi i (d/\lambda)^2  \sqrt{1-(\lambda n/d)^2} \,  \zeta }\,e^{2\pi i n \xi}.
\end{equation}
The quotient  $d/\lambda$ is a natural parameter in this problem, so let us call $r=d/\lambda$ and rewrite \eqref{Candidate_For_Convergence} as
\begin{equation}\label{Definition_Of_wr}
w_r(\xi, \zeta) = e^{-2\pi i r^2\zeta} \, \sum_{n \in \mathbb{Z}} e^{ 2\pi i r^2  \sqrt{1- (n/r)^2} \,  \zeta }\,e^{2\pi i n \xi}, \qquad \zeta  \geq 0.
\end{equation}
As with $u$ in \eqref{Helmholtz_Solution}, when $\zeta<0$ we may set $w_r(\xi,\zeta) = w_r(\xi,|\zeta|)$. 

The paraxial approximation \eqref{Paraxial_Approximation_Prelim} is valid for $|n| \ll r$, but not for larger values $|n| \gtrsim r$. As noted by Oskolkov \cite{Oskolkov2006}, this suggests that the paraxial approximation would be correct for data with frequencies smaller than $r$, or smooth enough to be approximated by such frequencies with a small error. However, this does not  apply to \eqref{Definition_Of_wr} because it corresponds to the Dirac comb on the boundary, which we do not expect to approximate only with small frequencies in view of the Poisson sumamation formula \eqref{Dirac_Comb}. Thus, in this setting it is reasonable to ask for $r \to \infty$. Indeed, we are going to prove that if $r \to \infty$ the Schr\"odinger solution \eqref{Free_Schrodinger_Solution} can be rigorously obtained from the Helmholtz solution \eqref{Definition_Of_wr}. 


Even if the Helmholtz solution \eqref{Definition_Of_wr} is a function, $v$ in \eqref{Free_Schrodinger_Solution} is not and has to be treated as a distribution whose action in the space of Schwartz functions $\mathcal S$ is
\begin{equation}
\langle v, \varphi \rangle = \sum_{n \in \mathbb Z} e^{-i\pi n^2 \zeta} \widehat{\varphi}(n), \qquad \forall \varphi \in \mathcal S(\mathbb R). 
\end{equation}
Thus, since we cannot expect convergence in the sense of functions, the results that we seek are in a distributional sense. We recall that if $\mathcal S'$ is the space of tempered distributions, a sequence  $T_n \in \mathcal{S}'$ is said to converge to $T \in \mathcal{S}'$ if 
\begin{equation}\label{Definition_Of_Convergence_Distributions}
\lim_{n \to \infty} \langle T_n, \varphi \rangle = \langle T, \varphi \rangle, \qquad \forall \varphi \in \mathcal{S}.
\end{equation}
We analyze convergence on different types of lines in the spirit of the fractal conjectures of Berry and Klein \cite{Berry1996,BerryKlein1996}.

\subsection{Convergence on horizontal lines}
The Talbot effect is a phenomenon that happens in the $\xi$ variable, for $\zeta$ is fixed (as can be observed in Figure~\ref{Figure_TalbotCarpet}). Thus, we consider  the solutions \eqref{Free_Schrodinger_Solution} and \eqref{Candidate_For_Convergence} as distributions in $\mathcal S'(\mathbb R)$ in the variable $\xi$.
Alternatively, both $w_r(\cdot,\zeta)$ and  $v(\cdot,\zeta)$ are 1-periodic Fourier series, so they can be regarded as periodic distributions in $\mathbb{R}$.
Denoting the torus by $\mathbb T = \mathbb R / \mathbb Z$, let $\mathcal{P}(\mathbb{T})$ be the space of smooth 1-periodic functions and $\mathcal{P}'(\mathbb{T})$ the space of 1-periodic distributions. In this setting, the action of $v$ is 
\begin{equation}
\langle v, \varphi \rangle = \sum_{n \in \mathbb Z} e^{-i\pi n^2 \zeta} \widehat \varphi_n, \qquad \forall \varphi \in \mathcal P(\mathbb T),
\end{equation}
where $\widehat \varphi_n$ is the $n$-th Fourier coefficient of $\varphi$. This allows us to be more precise in terms of Sobolev spaces
\begin{equation}\label{Sobolev_Spaces_In_T}
H^s(\mathbb{T}) = \left\{ \, f \in L^2(\mathbb{T}) \quad \mid \quad \lVert f \rVert_{H^s(\mathbb{T})}^2 = \sum_{n \in \mathbb{Z}} (1+n^2)^s\, | \widehat{f}_n  |^2 < \infty  \, \right\}, \qquad \forall s \geq 0,
\end{equation} 
which can be generalized to $s \in \mathbb{R}$ if instead of $f \in L^2(\mathbb{T})$ we consider $f \in \mathcal{P}'(\mathbb{T})$. Moreover, the duality $\left(H^s(\mathbb{T})\right)' = H^{-s}(\mathbb{T})$ holds.

We prove convergence in the two setting presented above.

\begin{thmMain}\label{THEOREM_Convergence_Tempered_And_Periodic_Distributions}
Let $\zeta >0$. Let $w_r = w_r(\cdot,\zeta)$ defined in \eqref{Definition_Of_wr} be considered as a function of $\xi$ alone, and $v= v(\cdot,\zeta)$ defined in \eqref{Free_Schrodinger_Solution} as a distribution of $\xi$ alone. Then, 
\begin{equation}
\lim_{r \to \infty} w_r(\cdot,\zeta) = v(\cdot,\zeta) \qquad\qquad \text{ in } \quad \mathcal{S}'(\mathbb{R}).
\end{equation}
If $ v(\cdot,\zeta)$ is considered as a periodic distribution, then for any $s > 1/2$ we have 
\begin{equation}
\lim_{r \to \infty} w_r(\cdot,\zeta) = v(\cdot,\zeta) \qquad\qquad \text{ in } \quad H^{-s}(\mathbb{T}).
\end{equation}
In both cases, if $\zeta<0$, then $w_r(\cdot,\zeta)=w_r(\cdot,|\zeta|)$, so $\lim_{r \to \infty}w_r(\cdot,\zeta)=v(\cdot,|\zeta|)$.
\end{thmMain}

\subsection{Convergence on vertical and oblique lines}
In view of Berry and Klein's fractal conjectures,
we should also pay attention to vertical and oblique lines. For instance, fixing $\xi \in \mathbb R$, the action of $v(\xi,\cdot) \in \mathcal S'(\mathbb{R})$  is
\begin{equation}
\langle v(\xi,\cdot),\varphi \rangle = \sum_{n \in \mathbb Z} e^{2\pi i n \xi}\, \widehat{\varphi}(n^2/2),\qquad \forall \varphi \in \mathcal S(\mathbb{R}).
\end{equation}
In  the case of oblique lines $\zeta = m\xi - k$ for $m,k\in\mathbb R$, the distribution $v \in \mathcal S'(\mathbb{R})$ in the variable $\xi$ is 
\begin{equation}
\langle v(\xi, m\xi - k), \varphi(\xi)\rangle = \sum_{n \in \mathbb Z} e^{i\pi n^2 k} \, \widehat{\varphi}\left(\frac{mn^2}{2} - n\right), \qquad \forall \varphi \in \mathcal S(\mathbb{R}).
\end{equation}
The study of $\zeta>0$ best represents the original Talbot's experiments, but we also prove convergence for $\zeta \in \mathbb R$.
The results analogous to Theorem~\ref{THEOREM_Convergence_Tempered_And_Periodic_Distributions} are as follows. 

\begin{thmMain}\label{THEOREM_Vertical}
Let $\xi \in\mathbb R$. Let $w_r = w_r(\xi,\cdot)$ defined in \eqref{Definition_Of_wr} be considered as a function of $\zeta>0$ alone, and $v = v(\xi,\cdot)$ defined in \eqref{Free_Schrodinger_Solution} as a distribution of $\zeta>0$ alone. Then, 
\begin{equation}
\lim_{r \to \infty} w_r(\xi,\cdot)\, \mathbbm 1_{(0,\infty)} = v(\xi,\cdot) \, \mathbbm 1_{(0,\infty)} \qquad\qquad \text{ in } \quad \mathcal{S}'(\mathbb{R}).
\end{equation}
If $w_r(\xi,\zeta)$ is given for all $\zeta \in \mathbb R$ by $w_r(\xi,\zeta)=w_r(\xi,|\zeta|)$, then $\lim_{r \to \infty} w_r(\xi,\cdot) = v(\xi,|\cdot|)$ in $\mathcal{S}'(\mathbb{R})$.
\end{thmMain}

\begin{thmMain}\label{THEOREM_Oblique}
Let $m >0$ and $k \in \mathbb R$. Let $w_r$ be defined in \eqref{Definition_Of_wr} and $v$ defined in \eqref{Free_Schrodinger_Solution}. Then, as distributions in the variable $\xi$ in the upper half plane,
\begin{equation}
\lim_{r \to \infty} w_r(\xi,m\xi - k) \, \mathbbm 1_{(k/m,\infty)} = v(\xi,m\xi - k) \, \mathbbm 1_{(k/m,\infty)} \qquad\qquad \text{ in } \quad \mathcal{S}'(\mathbb{R}).
\end{equation}
If $m<0$, the analogous result holds with the characteristic function $ \mathbbm 1_{(-\infty,-k/m)} $ instead. If the distributions are considered in the whole plane with $\xi \in \mathbb R$, then $\lim_{r \to \infty} w_r(\xi,m\xi - k) = v(\xi,|m\xi - k|)$  in $\mathcal{S}'(\mathbb{R})$.
\end{thmMain}

Since the Helmholtz solution \eqref{Candidate_For_Convergence} is not periodic in $\zeta$ nor along other lines, we have no analogue to the periodic result in Theorem~\ref{THEOREM_Convergence_Tempered_And_Periodic_Distributions}.

\subsection*{Structure of the article}
The rest of the article is devoted to prove Theorems~\ref{THEOREM_Convergence_Tempered_And_Periodic_Distributions} to \ref{THEOREM_Oblique}. The proofs follow the idea that in the Helmholtz solution \eqref{Candidate_For_Convergence} pointwise convergence holds in the frequency ranges $|n| \ll r$ and $|n| > r+1$, but not in the intermediate range $|n| \simeq r$ where the paraxial approximation fails. We explain this strategy in Section~\ref{SECTION_Strategy}. In Section~\ref{SECTION_Pointwise} we prove the part of the pointwise convergence and in Sections~\ref{SECTION_Tempered} and \ref{SECTION_Periodic} we cover the non-paraxial range for Theorem~\ref{THEOREM_Convergence_Tempered_And_Periodic_Distributions}. The proofs of Theorems~\ref{THEOREM_Vertical} and \ref{THEOREM_Oblique} follow the same strategy and are briefly given in Section~\ref{SECTION_VerticalAndOblique}.

\section{General strategy}\label{SECTION_Strategy}

Let us begin analyzing the structure of the Helmholtz solution $w_r(\xi,\zeta)$ in \eqref{Definition_Of_wr}. Regarding the applicability of the paraxial approximation 
\begin{equation}\label{Paraxial_Approximation}
\sqrt{1-(n/r)^2} \approx 1-\frac{(n/r)^2}{2},
\end{equation}
it is clear that there are three different ranges for $n$:
\begin{itemize}
	\item the indices $|n| \ll r$ represent the waves that are paraxial: \eqref{Paraxial_Approximation} can be used here.
	
	\item when $|n| \leq r$, $ |n| \simeq r$, the waves are not paraxial and are out of the range of validity of \eqref{Paraxial_Approximation}.
	
	\item when $|n| > r$, the \textit{waves} are not oscillating and decay exponentially. 
\end{itemize} 
To separate them, let us define a function $\mu(r)$ that satisfies 
\begin{equation}\label{Definition_Of_Mu}
\lim_{r\to\infty}\mu(r) = + \infty \qquad \text{ and } \qquad \lim_{r\to\infty} \mu(r)/r=0.
\end{equation}
The second condition represents the more informal $\mu(r) \ll r$ that we use to distinguish the two first cases above. With this,  let
\begin{equation}\label{Splitting_Of_w}
w_r = P_{\ll r}(w_{r}) + P_{\simeq r}(w_{r}) + P_{> r}(w_{r}),
\end{equation}
where $ P_{\ll r}(w_{r}),P_{\simeq r}(w_{r})$ and $P_{> r}(w_{r})$ are low, band and high-pass filters of $w_r$,
\begin{equation}\label{Definition_Of_w1}
P_{\ll r}(w_{r})(\xi,\zeta) =  e^{-2\pi i r^2\zeta} \, \sum_{|n| \leq \mu(r)} e^{ 2\pi i r^2   \sqrt{1- (n/r)^2} \,  \zeta  }\,e^{2\pi i n \xi},
\end{equation}
\begin{equation}\label{Definition_Of_w2}
 P_{\simeq r}(w_{r})(\xi,\zeta) =  e^{-2\pi i r^2\zeta} \, \sum_{\mu(r) < |n| \leq r+1} e^{ 2\pi i r^2   \sqrt{1- (n/r)^2} \,  \zeta  }\,e^{2\pi i n \xi},
\end{equation}
\begin{equation}\label{Definition_Of_w3}
P_{> r+1}(w_{r})(\xi,\zeta) =  e^{-2\pi i r^2\zeta} \, \sum_{|n| > r+1} e^{ -2\pi r^2   \sqrt{ (n/r)^2  - 1 }  \,  \zeta }\,e^{2\pi i n \xi}.
\end{equation}
We expect to get $v$ from $P_{\ll r}(w_{r})$ using the paraxial approximation. Regarding $P_{> r+1}(w_{r})$, it does not oscillate in $\zeta$ because it is a sum of decaying real exponentials. We will prove that it converges to zero pointwise. This choice instead of the a priori more natural $P_{> r}(w_{r})$ is technical, since one cannot prove the same pointwise convergence for the latter. The remaining term $P_{\simeq r}(w_r)$ is the most problematic one. Indeed, it is outside the range of validity of the paraxial approximation, so we will not be able to recover $v$ from it. Also, it oscillates so it does not decay. This is the part where distribution theory is needed. 

 In the coming sections, we will often drop the parentheses and write, for instance, $P_{\ll r}w_{r}$ for simplicity. We first tackle $P_{\ll r}w_{r}$ and $P_{> r+1}w_{r}$ in Section~\ref{SECTION_Pointwise}, and later $P_{\simeq r}w_{r}$ in Sections~\ref{SECTION_Tempered} and \ref{SECTION_Periodic}.

\section{Partial pointwise convergence}\label{SECTION_Pointwise}

In this section, we treat the pointwise convergence of $P_{\ll r}w_{r}$ and $P_{> r+1}w_{r}$. 
Let us start with  $P_{\ll r}w_{r}$. For that, let
\begin{equation}\label{Low_Pass_Filter_Of_v}
P_{\ll r}v(\xi,\zeta) = \sum_{|n| \leq \mu(r)} e^{2\pi i n \xi - i \pi n^2\zeta}
\end{equation}
and let us measure the error 
\begin{equation}\label{Error_Preliminary}
P_{\ll r}w_r(\xi,\zeta) - P_{\ll r}v(\xi,\zeta) = \sum_{|n| \leq \mu(r)} \left( e^{-2\pi i r^2\zeta} \, e^{ 2\pi i r^2   \sqrt{1- (n/r)^2} \, \zeta } - e^{- i\pi n^2 \zeta} \right)\, e^{2\pi i n \xi}
\end{equation}
when $r \to \infty$. Due to the fact that $|n| \leq \mu(r)$ and \eqref{Definition_Of_Mu}, we can use 
\begin{equation}
\sqrt{1-\left( \frac{n}{r}\right)^2} = 1 - \frac{n^2}{2r^2} + O\left( \frac{n^3}{r^3} \right)
\end{equation}
when $r$ is large enough, so from \eqref{Error_Preliminary} we get 
\begin{equation}\label{Absolute_Error_w1}
\begin{split}
\left| P_{\ll r}w_r(\xi,\zeta) - P_{\ll r}v(\xi,\zeta) \right|  & = \left| \sum_{|n| \leq \mu(r)} \left( e^{ 2\pi i r^2  \zeta \, O\left( n^3/r^3 \right)} - 1 \right) e^{-i \pi n^2 \zeta}\, e^{2\pi i n \xi} \right| \\
& \leq \sum_{|n| \leq \mu(r)}   \left| e^{ 2\pi i r^2  \zeta \,  O\left( n^3/r^3 \right)} - 1  \right|  \\
& \lesssim \sum_{n=1}^{\mu(r)}\frac{n^3}{r} \lesssim \frac{\mu(r)^4}{r},
\end{split}
\end{equation}
where we used the Taylor expansion of the exponential. For that, we need to ask $n \leq \mu(r) \ll r^{1/3}$. Moreover, this error converges to zero when $r\to\infty$ under the stronger condition
\begin{equation}\label{Condition_For_Mu_Low}
 \lim_{r\to\infty} \frac{\mu(r)}{r^{1/4}} = 0.
\end{equation}
We have thus proved a pointwise convergence result for the low-pass filter part. We write that in the following lemma:
\begin{lem}\label{Lemma_w1_Pointwise}
If $\mu$ defined in \eqref{Definition_Of_Mu} satisfies the condition \eqref{Condition_For_Mu_Low}, then
\begin{equation}
\lim_{r \to \infty} \lVert P_{\ll r}w_r - P_{\ll r}v  \rVert_{L^\infty_{\xi,\zeta}(\mathbb{R} \times (0,\infty))} = 0.
\end{equation}
\end{lem}

We now turn our attention to the highest frequencies $P_{> r+1 }w_r$. Their contribution is exponentially decaying and we are to show that it converges to zero pointwise when $r \to 0$. For that we need the following simple lemma:
\begin{lem}\label{Lemma_IntegralToZero}
Let $\zeta >0$. Then, 
\begin{equation}
\lim_{r\to\infty} \int_r^\infty e^{-2\pi r^2  \sqrt{(x/r)^2-1} \, \zeta}\, dx = 0.
\end{equation}
\end{lem}
\begin{proof}
Change variables $x=ry$ first and $r^2\,\sqrt{y^2-1} =z$ afterwards to get
\begin{equation}
\begin{split}
\int_r^\infty e^{-2\pi r^2  \sqrt{(x/r)^2-1}  \, \zeta }\, dx  & = \int_1^\infty e^{-2\pi r^2  \sqrt{y^2-1} \, \zeta}\, r\, dy =  \int_0^\infty  \frac{z\, e^{-4\pi z \zeta }}{r^3\,\sqrt{1+z^2/r^4}}\, dz.
\end{split}
\end{equation}
Call $f_r(z) = r^{-3}\, (1+z^2/r^4)^{-1/2} \,z \, e^{-4\pi z \zeta }$, so that when $r$ is large enough we have
\begin{equation}
|f_r(z)| \leq ze^{-4\pi z \zeta } \in L^1((0,\infty)).
\end{equation}
Also, for every fixed $z \geq 0$ we have $\lim_{r\to \infty} f_r(z) = 0$, so the theorem of dominated convergence gives the result.
\end{proof}
Pointwise convergence for $P_{> r+1 }w_r$ is a direct consequence of Lemma~\ref{Lemma_IntegralToZero}.
\begin{lem}\label{Lemma_PointwiseConvergenceHigh}
\begin{equation}
\lim_{r \to \infty} \lVert P_{> r+1}w_r \rVert_{L^\infty_{\xi,\zeta}(\mathbb{R} \times (0,+\infty))} = 0.
\end{equation}
\end{lem}
\begin{proof}
Given $\xi \in \mathbb R$ and $\zeta>0$, since the exponential function in Lemma~\ref{Lemma_IntegralToZero} is decreasing in $x$, we may simply write
\begin{equation}
\begin{split}
\lim_{r\to\infty} P_{> r+1}w_r(\xi, \zeta)  & = \lim_{r\to\infty} \sum_{n > r+1} e^{ -2\pi r^2  \sqrt{(n/r)^2 -1 } \, \zeta } \\ & \leq \lim_{r\to\infty} \int_r^\infty e^{-2\pi r^2  \sqrt{(x/r)^2 - 1} \, \zeta}\, dx = 0
\end{split}
\end{equation}
as a consequence of Lemma~\ref{Lemma_IntegralToZero}.
\end{proof}
\begin{rmk*}
Lemma~\ref{Lemma_PointwiseConvergenceHigh} is not true for $P_{> r}w_r$ because of the term $n=\lfloor r+1 \rfloor$. Indeed, assume that $r \in (M-1,M)$ for $M \in \mathbb N$. Thus, $\lfloor r+1 \rfloor = M$, and if $r \to M$ we have  $(\lfloor r+1 \rfloor/r)^2 - 1 \to 0$ and $e^{ -2\pi r^2  \sqrt{(\lfloor r+1 \rfloor/r)^2 -1 }  \, \zeta } \to 1$. Therefore, when $r \to \infty$ this happens around every integer, so $\limsup_{r \to \infty} e^{ -2\pi r^2  \sqrt{(\lfloor r+1 \rfloor/r)^2 -1 }  \, \zeta } = 1$. 
\end{rmk*}

To tackle the mid-range, oscillating and non-paraxial $P_{\simeq r}w_r$, we need the support of distribution theory.

\section{Convergence as tempered distributions. First part of Theorem~\ref{THEOREM_Convergence_Tempered_And_Periodic_Distributions}}
\label{SECTION_Tempered}
Let us begin by recalling that locally integrable functions of slow growth, that is, functions $f \in L^1_{loc}(\mathbb R)$ such that $\lim_{\xi \to \infty} |\xi|^{-N} f(\xi) = 0$ for some $N \in \mathbb{N}$, define tempered distributions by 
\begin{equation}\label{Regular_Distribution}
\langle f, \varphi \rangle = \int_{\mathbb{R}} f(\xi)\, \varphi(\xi) \, d\xi, \qquad \forall \varphi \in \mathcal{S}(\mathbb R).
\end{equation}
In particular, if $n \in \mathbb N$ and $f_n(\xi) = e^{-2\pi i n \xi}$, then $\langle f_n, \varphi \rangle = \widehat{\varphi}(n)$, where $\widehat\varphi$ is the Fourier transform of $\varphi$. 

Let us study the problematic $P_{\simeq r}w_r$. It is enough to work with $\zeta \geq 0$ because of the symmetry $w_r(\cdot,\zeta)=w_r(\cdot,|\zeta|)$. Thus, $P_{\simeq r}w_r(\cdot,\zeta) \in \mathcal{S}'$ is a tempered distribution in the variable $\xi$ that sends every $\varphi \in \mathcal S(\mathbb R)$ to 
\begin{equation}\label{eq:w2_Image}
\begin{split}
\langle  P_{\simeq r}w_r(\cdot,\zeta), \varphi \rangle & = e^{-2\pi i r^2\zeta} \, \sum_{\mu(r) < |n| \leq r + 1} e^{ 2\pi i r^2   \sqrt{1-(n/r)^2} \, \zeta }\, \langle e^{2\pi i n \xi}  , \varphi \rangle \\
& = e^{-2\pi i r^2\zeta} \, \sum_{\mu(r) < |n| \leq r+1} e^{ 2\pi i r^2  \sqrt{1 - (n/r)^2}  \, \zeta}\, \widehat \varphi(-n).
\end{split}
\end{equation}
Consequently,
\begin{equation}
\left|  \langle  P_{\simeq r}w_r(\cdot,\zeta), \varphi \rangle \right|  = \Bigg| \sum_{\mu(r) < |n| \leq r+1} e^{ 2\pi i r^2   \sqrt{1-(n/r)^2} \, \zeta}\, \widehat{\varphi}(-n) \Bigg| 
\leq \sum_{\mu(r) < |n| \leq r+1} \left|   \widehat{\varphi}(n) \right|.
\end{equation}
Since $\varphi \in \mathcal{S}$ implies $\widehat{\varphi} \in \mathcal{S}$, the sequence $\widehat{\varphi}(n)$ decays faster than $n^{-2}$, and thus the sum above is the tail of a convergent series. Since $\lim_{r\to\infty}\mu(r) = \infty$, we can write
\begin{equation}\label{eq:Use_Decay_Of_Fourier_Transform}
\lim_{r \to \infty}\left|  \langle  P_{\simeq r}w_r(\cdot,\zeta), \varphi \rangle \right|   \leq   C_\varphi \,  \lim_{r \to \infty}  \sum_{n = \mu(r) }^\infty \frac{ 1}{n^2} = 0, \qquad \forall \varphi \in \mathcal S(\mathbb R),
\end{equation}
for some constant $C_\varphi>0$ that depends on $\varphi$. So we have proved the following:
\begin{lem}\label{Lemma_w2_Distributions}
Let $\zeta \in \mathbb R$ and $w_r(\xi,\zeta)=w_r(\xi,|\zeta|)$. Then, 
\begin{equation}
\lim_{r \to \infty} P_{\simeq r}w_r(\cdot,\zeta) = 0   \qquad \qquad \text{ in } \quad  \mathcal{S}'(\mathbb R).
\end{equation}
\end{lem}

Of course, from Lemmas~\ref{Lemma_w1_Pointwise} and \ref{Lemma_PointwiseConvergenceHigh} we easily deduce the distributional convergence for $P_{\ll r}w_r$ and $P_{>r+1}w_r$.
\begin{lem}\label{Lemma_DistributionsHighAndLow}
Let $\zeta  \in \mathbb R$. If $\mu(r)$ satisfies \eqref{Condition_For_Mu_Low}, then 
\begin{equation}
\lim_{r \to \infty} \Big( P_{\ll r}w_r(\cdot, |\zeta|) - P_{\ll r}v(\cdot, |\zeta|) \Big) = 0 \qquad \text{ in } \quad \mathcal{S}'(\mathbb R).
\end{equation}
Also, 
\begin{equation}
\lim_{r \to \infty}  P_{> r+1}w_r(\cdot, |\zeta|)  = 0 \qquad \text{ in } \quad \mathcal{S}'(\mathbb R).
\end{equation}
\end{lem}
\begin{proof}
Call $f_r(\xi,\zeta)$ to either of the functions in the statement. Then, for every $\zeta>0$ and any $\varphi \in \mathcal S (\mathbb R)$, by Lemmas~\ref{Lemma_w1_Pointwise} and \ref{Lemma_PointwiseConvergenceHigh} we have
\begin{equation}
\begin{split}
\lim_{r\to \infty} \left|  \langle f_r(\cdot,\zeta), \varphi \rangle \right| & \leq \lim_{r\to \infty} \int_{\mathbb R} \left| f_r(\xi,\zeta)\, \varphi(\xi) \right|\, d\xi \\ 
& \leq \lVert \varphi \rVert_{L^1(\mathbb R)} \lim_{r\to \infty} \lVert f_r \rVert_{L^{\infty}_{\xi,\zeta}(\mathbb R \times (0,\infty)))} = 0.
\end{split}
\end{equation}
\end{proof}

We are now ready to prove the first part of Theorem~\ref{THEOREM_Convergence_Tempered_And_Periodic_Distributions}.
\begin{proof}[Proof of first part of Theorem~\ref{THEOREM_Convergence_Tempered_And_Periodic_Distributions}]
Let $\zeta \in  \mathbb R$. Decompose $w_r$ as in \eqref{Splitting_Of_w} and write
\begin{equation}\label{DecompositionOfWr}
\begin{split}
(w_r - v)(\cdot, |\zeta|) &  = \left( P_{\ll r}w_r - P_{\ll r}v \right)(\cdot,|\zeta|) + P_{\simeq r}w_r(\cdot,|\zeta|) \\
& \quad + P_{>r+1}w_r(\cdot,|\zeta|) - \left( v - P_{\ll r}v \right)(\cdot,|\zeta|).
\end{split}
\end{equation}
By Lemmas~\ref{Lemma_w2_Distributions} and \ref{Lemma_DistributionsHighAndLow}, we only need to prove that $(v-P_{\ll r}v)(\cdot,\zeta)$ tends to zero when $r \to \infty$. For that, let $\zeta \in \mathbb R$ and $\varphi \in \mathcal S(\mathbb R)$, and write
\begin{equation}
\left| \langle (v - P_{\ll r}v)(\cdot,\zeta), \varphi  \rangle \right|= \Bigg| \sum_{|n| > \mu(r)} e^{-i \pi n^2 \zeta}  \langle e^{2\pi i n \xi},\varphi \rangle \Bigg| \leq \sum_{|n| > \mu(r)} \left|  \widehat{\varphi}(n) \right|.
\end{equation}
Since $\widehat\varphi \in \mathcal S(\mathbb R)$, the sum is the tail of a convergent series, so it tends to zero when $r \to \infty$ because $\lim_{r \to \infty}\mu(r) = \infty$. 
\end{proof}

\section{Convergence as periodic distributions. Second part of Theorem~\ref{THEOREM_Convergence_Tempered_And_Periodic_Distributions}}
\label{SECTION_Periodic}
In this section, we look at $w_r$ and $v$ as 1-periodic distributions in the variable $\xi$. For the sake of completeness, we begin with a few words on periodic distributions. Let $\mathcal{P}(\mathbb T)$ be the space of 1-periodic smooth test functions. There is an identification between the usual distributions in $\mathcal{D}'(\mathbb{R})$ that are 1-periodic and the dual space $(\mathcal{P}(\mathbb T))'$, the reader can check the details of the general theory in \cite[Chapter 11]{Zemanian1965}. Here, it will be enough to focus on regular distributions that are defined by functions $f\in L^1_{loc}(\mathbb R)$  of slow growth and the expression in \eqref{Regular_Distribution}. If $f$ is 1-periodic, then it is identified with an element $f_P \in (\mathcal P(\mathbb T))'$, which for simplicity we will just call $f$, by means of 
\begin{equation}
\langle f, \varphi \rangle = \int_{\mathbb{T}} f(\xi)\, \varphi(\xi)\, d\xi, \qquad \forall \varphi \in \mathcal{P}(\mathbb T).
\end{equation}
For instance, if $f_n(\xi) = e^{-2\pi i n \xi}$, we saw that $f_n \in \mathcal D'(\mathbb R)$. Thanks to the identification above, we also have $f_n \in (\mathcal P(\mathbb T))'$ by means of
\begin{equation}\label{Exponentials_Periodic}
\langle f_n,\varphi \rangle = \int_{\mathbb T} \varphi(\xi)\, e^{-2\pi i n \xi}\, d\xi = \widehat\varphi_n, \qquad \forall \varphi \in \mathcal P(\mathbb T),
\end{equation}
where $\widehat\varphi_n$ is the $n$-th Fourier coefficient of $\varphi$.

Having said this, let $\zeta \geq 0$ and take the problematic non-paraxial portion of the Helmholtz solution $P_{\simeq r}w_r$, which according to \eqref{Exponentials_Periodic} can be treated as an element of $(\mathcal P(\mathbb T))'$. More precisely, 
\begin{equation}
\begin{split}
\left| \langle P_{\simeq r}w_r(\cdot,\zeta), \varphi  \rangle \right| &  = \Bigg| \sum_{\mu(r) < |n| \leq r+1} e^{ 2\pi i r^2   \sqrt{1-(n/r)^2}  \, \zeta}\, \langle e^{2\pi i n \xi}  , \varphi \rangle \Bigg| \\ 
& \leq \sum_{\mu(r) < |n| \leq r + 1}    \left| \widehat{\varphi}_n  \right|, 
\end{split}
\end{equation}
for all $\varphi \in \mathcal P(\mathbb T)$.
In this periodic context we can be more precise on the decay needed for $\varphi$ in terms of the Sobolev spaces defined in \eqref{Sobolev_Spaces_In_T}. Indeed, by the Cauchy-Schwartz inequality we may write
\begin{equation}\label{eq:Varphi_Is_Summable}
\begin{split}
\sum_{\mu(r) < |n| \leq r + 1}    \left| \widehat{\varphi}_n  \right|  & \leq \left( \sum_{\mu(r) < |n| \leq r+1}    \left| \widehat{\varphi}_n  \right|^2 \,  (1+n^2)^s  \right)^{1/2}  \left( \sum_{\mu(r) < |n| \leq r+1}    \frac{1}{(1+n^2)^s} \right)^{1/2} \\
& \lesssim \lVert \varphi \rVert_{H^s(\mathbb{T})} \, \left( \sum_{n =\mu(r) }^\infty    \frac{1}{n^{2s}} \right)^{1/2}.
\end{split}
\end{equation}
As long as $2s>1$, the last series is the tail of a convergent harmonic series, so it tends to zero when $r \to \infty$. In other words, convergence to zero holds if we test against functions in $H^s(\mathbb T)$ for $s>1/2$. Thus, we have the following result. 
\begin{lem}\label{Lemma_w2_Periodic}
Let  $s>1/2$ and $\zeta \in  \mathbb R$. Then, 
\begin{equation}
\lim_{r \to \infty} P_{\simeq r}w_r(\cdot,|\zeta|) = 0  \qquad \text{ in } \quad H^{-s}(\mathbb{T}).
\end{equation}
\end{lem}

As in the previous section, the pointwise convergence of $P_{\ll r}w_r$ and $P_{>r+1}w_r$ implies the weaker convergence in $H^{-s}(\mathbb T)$.
\begin{lem}\label{Lemma_PeriodicDistributionsHighAndLow}
Let $s \geq 0$ and $\zeta \in \mathbb R$. If $\mu(r)$ satisfies \eqref{Condition_For_Mu_Low}, then 
\begin{equation}
\lim_{r \to \infty} \Big( P_{\ll r}w_r(\cdot, |\zeta|) - P_{\ll r}v(\cdot, |\zeta|) \Big) = 0 \qquad \text{ in } \quad H^{-s}(\mathbb T).
\end{equation}
Also, 
\begin{equation}
\lim_{r \to \infty}  P_{> r+1}w_r(\cdot, |\zeta|)  = 0 \qquad \text{ in } \quad H^{-s}(\mathbb T).
\end{equation}
\end{lem}
\begin{proof}
If $s \geq 0$, for $\varphi \in H^s(\mathbb T)$ we have $\lVert \varphi \rVert_{L^1(\mathbb T)} \leq \lVert \varphi \rVert_{L^2(\mathbb T)} \leq \lVert \varphi \rVert_{H^s(\mathbb T)}$. Thus, if $f_r(\xi,\zeta)$ is any of the functions in the statement, by Lemmas~\ref{Lemma_w1_Pointwise} and \ref{Lemma_PointwiseConvergenceHigh} we get
\begin{equation}
\begin{split}
\lim_{r\to \infty} \left|  \langle f_r(\cdot,\zeta), \varphi \rangle \right| & \leq \lim_{r\to \infty} \int_{\mathbb T} \left| f_r(\xi,\zeta)\, \varphi(\xi) \right|\, d\xi \\ 
& \leq \lVert \varphi \rVert_{H^s(\mathbb T)} \lim_{r\to \infty} \lVert f_r \rVert_{L^{\infty}_{\xi,\zeta}(\mathbb R \times (0,\infty))} = 0.
\end{split}
\end{equation}
\end{proof}

From these lemmas the proof of the second part of Theorem~\ref{THEOREM_Convergence_Tempered_And_Periodic_Distributions} follows immediately.
\begin{proof}[Proof of second part of Theorem~\ref{THEOREM_Convergence_Tempered_And_Periodic_Distributions}]
Decompose $w_r$ like in \eqref{DecompositionOfWr} so that by Lemmas~\ref{Lemma_w2_Periodic} and \ref{Lemma_PeriodicDistributionsHighAndLow} it is enough to prove that $(v - P_{\ll r}v)(\cdot,\zeta)$ tends to zero when $r \to \infty$. Now, given $\varphi \in H^s(\mathbb T)$, 
\begin{equation}
\begin{split}
\left| \langle (v - P_{\ll r}v)(\cdot,\zeta), \varphi  \rangle \right| & = \Bigg| \sum_{|n| > \mu(r)} e^{-i \pi n^2 \zeta}  \langle e^{2\pi i n \xi},\varphi \rangle \Bigg| \\
&  \leq \sum_{|n| > \mu(r)} \left|  \widehat{\varphi}_n \right| \leq \lVert \varphi \rVert_{H^s(\mathbb T)} \left(  \sum_{n=\mu(r)}^\infty \frac{1}{n^{2s}} \right)^{1/2},
\end{split}
\end{equation}
which, as before, tends to zero if $2s>1$. 
\end{proof}

\section{Convergence along vertical and oblique lines: Theorems~\ref{THEOREM_Vertical} and \ref{THEOREM_Oblique}}
\label{SECTION_VerticalAndOblique}

We follow the strategy for Theorem~\ref{THEOREM_Convergence_Tempered_And_Periodic_Distributions}  to prove Theorems~\ref{THEOREM_Vertical} and \ref{THEOREM_Oblique}. 

\begin{proof}[Proof of Theorem~\ref{THEOREM_Vertical}]
Let $\xi \in \mathbb R$. We first prove the result for $w_r(\xi, \cdot) \, \mathbbm{1}_{(0,\infty)}$ following what we did for Theorem~\ref{THEOREM_Convergence_Tempered_And_Periodic_Distributions} with the decomposition \eqref{DecompositionOfWr}. Indeed, the analogue of Lemma~\ref{Lemma_DistributionsHighAndLow} is an immediate consequence of Lemmas~\ref{Lemma_w1_Pointwise} and \ref{Lemma_PointwiseConvergenceHigh}. Also, for $\varphi \in \mathcal S(\mathbb R)$,
\begin{equation}
\begin{split}
\left|  \langle v(\xi,\cdot)\, \mathbbm{1}_{(0,\infty)} - P_{\ll r}[ v(\xi,\cdot)\, \mathbbm{1}_{(0,\infty)}], \varphi \rangle \right| & \leq \sum_{|n| > \mu(r)} \left| \int_0^\infty \varphi(\zeta)\, e^{-i\pi n^2 \zeta}  d\zeta \right| \\
& \lesssim  \sum_{|n| > \mu(r)}\frac{1}{n^2},
\end{split}
\end{equation}
which tends to zero when $r\to \infty$. The last inequality holds because for every $x \in \mathbb R$ we have
\begin{equation}\label{DecayOfTruncatedFourier}
\left| -2\pi i x\, \int_0^\infty \varphi(\zeta)\, e^{-2\pi i \zeta x}\, d\zeta \right| = \left| \varphi(0) - \int_0^\infty \varphi'(\zeta) \, e^{-2\pi i \zeta x}\, d\zeta \right| \leq \lVert \varphi \rVert_{\infty} + \lVert \varphi' \rVert_1.
\end{equation}
Therefore, it is enough to prove that $\lim_{r\to \infty}P_{\simeq r}[w_r\, \mathbbm 1_{(0,\infty)}](\xi,\cdot) = 0$ in $\mathcal{S}'(\mathbb R)$. 
For $\varphi \in \mathcal S(\mathbb R)$, 
\begin{equation}
\begin{split}
\langle P_{\simeq r}[w_r\, \mathbbm 1_{(0,\infty)}](\xi,\cdot),\varphi \rangle &  = \sum_{\mu(r) < |n| < r+1} e^{2\pi i n \xi} \, \langle e^{-2\pi i r^2 \zeta}\, e^{2\pi i r^2  \sqrt{1 - (n/r)^2} \, \zeta}, \varphi  \mathbbm 1_{(0,\infty)}(\zeta) \rangle \\
&  = I + II
\end{split}
\end{equation}
where 
\begin{equation}
I =  \sum_{\mu(r) < |n|  \leq  r} e^{2\pi i n \xi}  \,
 \widehat{ \varphi  \mathbbm 1_{(0,\infty)}}\left( r^2 \left( 1 - \sqrt{1-(n/r)^2} \right) \right)
\end{equation}
and
\begin{equation}
II = e^{2\pi i  \lfloor r+1 \rfloor \xi} \, \langle e^{-2\pi i r^2 \zeta}\, e^{-2\pi r^2  \sqrt{( \lfloor r+1 \rfloor/r)^2 - 1}  \, \zeta}, \varphi  \mathbbm 1_{(0,\infty)}(\zeta) \rangle.
\end{equation}
Rewriting $r^2\left( 1 - \sqrt{1-(n/r)^2}\right) =n^2 /  \left( 1 + \sqrt{1-(n/r)^2}\right)$, by \eqref{DecayOfTruncatedFourier} we get 
\begin{equation}\label{OneInProof}
\begin{split}
|I| & \leq \sum_{\mu(r) < |n|  \leq  r} \left|  \widehat{ \varphi  \mathbbm 1_{(0,\infty)}}\left(  \frac{n^2}{  1 + \sqrt{1-(n/r)^2}} \right) \right| \lesssim \sum_{\mu(r) < |n|  \leq  r} \frac{  1 + \sqrt{1-(n/r)^2}}{n^2} \\
&  \lesssim \sum_{\mu(r) < |n|  \leq  r} \frac{  1}{n^2} ,
\end{split}
\end{equation}
which shows that $\lim_{r \to \infty}I = 0$. On the other hand, integrating by parts we get 
\begin{equation}\label{ByParts}
\begin{split}
|II| & = \left|  \int_0^\infty \varphi(\zeta) \, e^{-2\pi r^2  \sqrt{( \lfloor r+1 \rfloor/r)^2 - 1}  \, \zeta } \, e^{-2\pi i r^2 \zeta}\, d\zeta \right| \\
& \leq \frac{|\varphi(0)|}{2\pi r^2} +  \frac{1}{2\pi r^2}\,  \left|  \int_0^\infty  \varphi'(\zeta) \, e^{-2\pi r^2  \sqrt{ \left( \frac{\lfloor r+1 \rfloor}{r}\right)^2 - 1} \, \zeta } \, e^{-2\pi i r^2 \zeta}\, d\zeta \right| \\
& \qquad \qquad + \sqrt{ \left( \frac{\lfloor r+1 \rfloor}{r} \right)^2 - 1} \, \left|  \int _0^\infty \varphi(\zeta) \, e^{-2\pi r^2  \sqrt{ \left( \frac{\lfloor r+1 \rfloor}{r}\right)^2 - 1}  \, \zeta }  \,e^{-2\pi i r^2 \zeta} \, d\zeta \right|,
\end{split}
\end{equation}
so 
\begin{equation}
\lim_{r\to\infty} |II| \leq \lim_{r\to \infty}  \left(  \frac{\lVert \varphi' \rVert_1}{2\pi r^2} +  \lVert \varphi \rVert_1  \sqrt{ \left( \frac{\lfloor r+1 \rfloor}{r}\right)^2 - 1} \right) = 0,
\end{equation}
which concludes the first part of the theorem. Finally, if $w_r(\xi,\zeta)$ is considered for all $\zeta \in \mathbb R$, then $w_r(\xi,\zeta)=w_r(\xi,|\zeta|)$ and the proof is the same as above if we replace $\varphi(\zeta)$ with $\varphi(\zeta) + \varphi(-\zeta)$ in the computations. 
\end{proof}

We now prove Theorem~\ref{THEOREM_Oblique}.
\begin{proof}[Proof of Theorem~\ref{THEOREM_Oblique}]
Let $m>0$. The case $m<0$ is analogous, so we do not write it. 
For the same reasons as in the proof of Theorem~\ref{THEOREM_Vertical} above, \eqref{DecayOfTruncatedFourier} implies that 
\begin{equation}
\begin{split}
& \left|  \langle v(\xi,m\xi - k) \mathbbm{1}_{(k/m,\infty)} - P_{\ll r}[ v(\xi,m\xi - k) \mathbbm{1}_{(k/m,\infty)}], \varphi \rangle \right| \\
& \qquad \qquad  \qquad \qquad \qquad  \qquad   = \left|  \sum_{|n| > \mu(r)} \int_{k/m}^\infty \varphi(\xi)\, e^{2\pi i n \xi - i\pi n^2(m\xi - k)}\, d\xi \right|  \\
& \qquad \qquad  \qquad \qquad \qquad  \qquad \leq \sum_{|n| > \mu(r)}   \left|  \int_{k/m}^\infty \varphi(\xi)\, e^{-2\pi i \xi ( mn^2/2 - n) }\, d\xi   \right| \\
& \qquad \qquad \qquad \qquad \qquad  \qquad  \lesssim \sum_{|n| > \mu(r)}  \frac{1}{  |mn^2/2-n|} \to 0  \qquad \text{ as } r \to \infty, 
\end{split}
\end{equation}
for all $\varphi \in \mathcal S (\mathbb R)$.
Thus, it is enough to prove $\lim_{r\to\infty} P_{\simeq r}w_r(\xi,m\xi - k)\mathbbm{1}_{(k/m,\infty)} = 0$ in $\mathcal S'(\mathbb R)$. 
Given $\varphi \in \mathcal S(\mathbb R)$, 
\begin{equation}
\begin{split}
& \langle  P_{\simeq r}w_r(\xi,m\xi - k)\mathbbm{1}_{(k/m,\infty)}, \varphi     \rangle \\
&    \qquad  = \sum_{\mu(r) < |n| < r+1} e^{2\pi i r^2 k \left( 1 - \sqrt{1 - (n/r)^2} \right)} \langle  e^{2\pi i n \xi } e^{-2\pi i r^2 m \xi \left(1 - \sqrt{1-(n/r)^2} \right)}, \varphi \mathbbm{1}_{(k/m,\infty)}  \rangle \\
&    \qquad  = I + II,
\end{split}
\end{equation}
where $I$ corresponds to the sum in $\mu(r) < |n| \leq r$ and $II$ to $|n| = \lfloor r+1 \rfloor$. Now, like in \eqref{OneInProof},
\begin{equation}
\begin{split}
|I| & \leq \sum_{\mu(r) < |n| \leq r} \left|  (\varphi \mathbbm{1}_{(k/m,\infty)})^\wedge\left( -n + \frac{mn^2}{ 1 + \sqrt{1-(n/r)^2} } \right) \right| \\
& \lesssim_\varphi \sum_{\mu(r) < |n| \leq r} \frac{1}{\left| -n + \frac{mn^2}{ 1 + \sqrt{1-(n/r)^2} } \right|} \\
& \lesssim  \sum_{\mu(r) < |n| \leq r} \frac{1}{mn^2} \to 0 \qquad \text{ as } r \to \infty.
\end{split}
\end{equation}
On the other hand, with $n=\lfloor r+1 \rfloor$ we write
\begin{equation}
II = \int_{k/m}^\infty \varphi(\xi) \, e^{-2\pi (m\xi - k) r^2\sqrt{(n/r)^2-1}}\, e^{2\pi i n \xi - 2\pi i r^2(m\xi-k)}\, d\xi.
\end{equation}
Either integrating by parts as in \eqref{ByParts}, or by the theorem of dominated convergence because $r^2\sqrt{(n/r)^2 - 1} \leq r^2\sqrt{(r+1)^2/r^2-1} = r\sqrt{2r+1} \to \infty$, we get $\lim_{r\to \infty} II = 0$ and the proof of the first part is complete.

Finally, if for $\xi \in \mathbb R$  we write $w_r(\xi,m\xi-k) =  w_r(\xi,|m\xi-k|)$, then after changing variables and a few manipulations the proof essentially reduces to the procedure above replacing  $\varphi(\xi)$ with $\varphi(\xi) + \varphi(-\xi)$.
\end{proof}

\appendix

\section{The Talbot effect as the solution to the free Schr\"odinger equation}\label{AppendixTalbot}
For the sake of completeness, we briefly explain why the solution to the free Schr\"odinger equation \eqref{Free_Schrodinger_Solution} captures the most important features of the Talbot effect. For that we use Figure~\ref{Figure_TalbotCarpet} as a reference. We recall that $z = z_T \, \zeta$.

First of all, $v$ is periodic of period 2 in the vertical variable $\zeta$, in the same way that Figure~\ref{Figure_TalbotCarpet} shows the repetition of the grating at height $z = 2z_T$. But the Talbot carpet shows many more patterns that correspond to rational multiples $z = (p/q)z_T$. Indeed, evaluating $v$ at $\zeta=p/q$ where $p$ and $q$ are coprime integers, we may split the sum by $n = mq+r$ to write
\begin{equation}
\begin{split}
v(\xi,p/q) & = \sum_{n \in \mathbb Z} e^{2\pi i n \xi - i \pi n^2 p/q} = \sum_{m\in \mathbb Z} \sum_{r = 0}^{q-1} e^{2\pi i (mq+r)\xi - i\pi (mq+r)^2 p/q} \\
& = \sum_{r = 0}^{q-1}  e^{2\pi i r\xi - i \pi r^2 p/q}   \sum_{m\in \mathbb Z} e^{2\pi i mq \xi - i\pi m^2 pq}.
\end{split}
\end{equation}
The fact that $e^{i\pi m^2} = (-1)^{m^2} = (-1)^m$ for every $m \in \mathbb Z$ allows us to rewrite the sum in $m$ in such a way that the Poisson summation formula can be used, so that 
\begin{equation}\label{Talbot_By_Schrodinger}
\begin{split}
v(\xi,p/q) & =  \frac{1}{q}\,    \sum_{r = 0}^{q-1}  e^{2\pi i r\xi - i \pi r^2 p/q} \sum_{m\in \mathbb Z} \delta\left(\xi - \frac{p \, (\operatorname{mod} 2)}{2} - \frac{m}{q}\right) \\
& = \frac{1}{q}\,  \sum_{m\in \mathbb Z} \left(  \sum_{r = 0}^{q-1}  e^{2\pi i \frac{(qp (\operatorname{mod } 2)/2 + m) r - (p/2) r^2}{q} } \right)  \delta\left(\xi - \frac{p \, (\operatorname{mod} 2)}{2} - \frac{m}{q}\right) \\
& : =  \frac{1}{q}\,  \sum_{m\in \mathbb Z} \Gamma(p,q;m)  \delta\left(\xi - \frac{p \, (\operatorname{mod} 2)}{2} - \frac{m}{q}\right).
\end{split}
\end{equation}
We can observe that $\Gamma(p,q;m)$ is a variation of the classical generalized quadratic Gauss sums, and it satisfies $|\Gamma(p,q;m)| = \sqrt{q}$ for every $p,q$ and $m$. Depending on the values of $p$ and $q$, the steps needed to prove this are different, see Appendix~\ref{Section_GaussSums}.

Expression \eqref{Talbot_By_Schrodinger} shows a sum of Dirac deltas at $\zeta = p/q$ with a separation of $1/q$ and with an amplitude $1/\sqrt{q}$. This represents a copy of the grating with $q$ times as many slits and $\sqrt{q}$ times less intense. Also, if $p$ is odd, the copy is shifted half a period with respect to the initial grating. This is precisely what the Talbot carpet in Figure~\ref{Figure_TalbotCarpet} shows.

\section{A few remarks on generalized quadratic Gauss sums}\label{Section_GaussSums}
The standard generalized quadratic Gauss sums
\begin{equation}
G(a,b,c) = \sum_{r=0}^{c-1} e^{2\pi i \frac{ar^2+br}{c}}, \qquad a,b \in \mathbb Z, \, \,  c\in \mathbb N,
\end{equation}
have been vastly studied in the literature. For instance, it is well-known that $|G(a,b,c)| = \sqrt{c}$ for $c$ odd (for instance, completing the square and using \cite[Theorem 1.5.2]{BerndtEvansWilliams1998}). With that property, if $p$ is even in $\Gamma(p,q;m)$ we get $p\pmod{2}=0$ and $q$ odd and thus $|\Gamma(p,q;m)| = |G(-p/2,m,q)|=\sqrt{q}$. If $p$ is odd, by direct computation we have
\begin{equation}
\sum_{r=0}^{q-1} e^{2\pi i \frac{(q+2m)(r+q) - p(r+q)^2}{2q}} = \sum_{r=0}^{q-1} e^{2\pi i \frac{(q+2m)r - p\, r^2}{2q}} , 
\end{equation}
so 
\begin{equation}
\Gamma(p,q;m) = \sum_{r=0}^{q-1} e^{2\pi i \frac{ (q+2m)r - p\, r^2  }{2q}} = \frac12\, \sum_{r=0}^{2q-1}e^{2\pi i \frac{ (q+2m)r - p\,r^2  }{2q}} = \frac12\,  G(-p,2m+q,2q).
\end{equation}
With this and the multiplicative property $G(a,b,cd) = G(ac,b,d)G(ad,b,c)$ for coprime $c,d$ (\cite[Lemma 1.2.5]{BerndtEvansWilliams1998}), if $q$ is odd we can write
\begin{equation}
\begin{split}
|\Gamma(p,q;m) | & = \frac12 |G(-2p,2m+q,q)G(-pq,2m+q,2)| \\ 
& = \frac12 \, \sqrt{q}\,  |G(1,1,2)| =  \sqrt{q}.
\end{split}
\end{equation}
We are left with the case $p$ odd and $q$ even. In this case, since $p$ and $2q$ are coprime and $2m+q$ is even, completing the square and denoting $p^{-1}$ to be the inverse of $p$ modulo $2q$, we can write
\begin{equation}
G(-p,2m+q,2q) = e^{ -2\pi i p^{-1} (\frac{2m+q}{2})^2 }\, G(-p,0,2q).
\end{equation}
Since $2q \equiv 0 \text{ (mod } 4)$, the value of the quadratic Gauss sum is known (\cite[Theorem 1.5.4]{BerndtEvansWilliams1998}), so
\begin{equation}
|\Gamma(p,q;m) |  = \frac12 \,  |G(-p,0,2q)| = \frac12 \left|1+i^{-p}\right|\sqrt{2q} = \sqrt{q}.
\end{equation}

\begin{acknowledgements*}
This research is supported by the Simons Foundation Collaboration Grant on Wave Turbulence (Nahmod’s Award ID 651469).
\end{acknowledgements*}

\bibliographystyle{acm}
\bibliography{Talbot.bib}

\end{document}